\documentclass[11pt]{article}

\usepackage[utf8]{inputenc}
\usepackage[T1]{fontenc}
\usepackage{mathpazo}

\RequirePackage[english]{babel}

\RequirePackage[a4paper,margin=2.5cm]{geometry}
\RequirePackage{parskip}

\newcommand{\affiliation}{\footnote}
\newcommand{\affiliationmark}[1][\value{footnote}-1]{\footnotemark[\numexpr#1+1\relax]}
\usepackage{xcolor}
\definecolor{cblue}{RGB}{0,70,140}
\definecolor{cgreen}{RGB}{100,140,0}
\definecolor{cred}{RGB}{190,10,50}

\usepackage[shortlabels]{enumitem}
\setlist{itemsep=0ex,topsep=0ex,parsep=0.4ex}

\usepackage{amsmath,amsfonts,amssymb,amsthm,mathtools,thmtools,thm-restate,bbm,centernot}

\usepackage[hyphens]{url} 
\usepackage[hidelinks,colorlinks,pagebackref]{hyperref} 
\hypersetup{citecolor=cgreen,linkcolor=cblue,urlcolor=cblue}
\usepackage[capitalise,nameinlink,noabbrev,compress]{cleveref} 

\renewcommand*{\backref}[1]{}
\renewcommand*{\backrefalt}[4]{
	\ifcase #1 Not cited.%
	\or $\uparrow$#2%
	\else $\uparrow$#2%
	\fi%
}

\let\oldbibliography\bibliography
\renewcommand{\bibliography}[1]{
  {

    \hypersetup{linkcolor=cred}
    \bibliographystyle{bibstyle}
    \oldbibliography{#1}
  }
}

\theoremstyle{plain}
\newtheorem{theorem}{Theorem}[section]
\newtheorem{lemma}[theorem]{Lemma}

\newtheorem{proposition}[theorem]{Proposition}

\newtheorem{corollary}[theorem]{Corollary}
\newtheorem{conjecture}[theorem]{Conjecture}
\newtheorem{problem}[theorem]{Problem}

\theoremstyle{definition}

\makeatletter
\renewenvironment{proof}[1][\proofname]
{\par\pushQED{\qed}
	\normalfont\topsep6\p@\@plus6\p@\relax\trivlist
	\item[\hskip\labelsep\bfseries#1\@addpunct{.}]
	\ignorespaces}
{\popQED\endtrivlist\@endpefalse}
\makeatother

\newcommand{\pr}{\mathbb{P}}
\newcommand{\ev}{\mathbb{E}}
\newcommand{\ent}{\mathcal{H}}
\newcommand{\ind}{\mathbbm{1}}
\newcommand{\cI}{\mathcal I}
\newcommand{\cO}{\mathcal O}
\newcommand{\cL}{\mathcal L}
\newcommand{\C}{\mathcal C}
\newcommand{\X}{\mathcal X}
\newcommand{\Y}{\mathcal Y}
\newcommand{\logt}{\log_2}
\DeclarePairedDelimiter{\abs}{\lvert}{\rvert}

\title{Cycle-factors of regular graphs via entropy}
\author{Micha Christoph\affiliation{Department of Mathematics, ETH Z\"{u}rich, Switzerland (\textsf{\{\href{mailto:micha.christoph@math.ethz.ch}{micha.christoph},\href{mailto:lukas.michel@math.ethz.ch}{lukas.michel}\}@math.ethz.ch}). Research of Micha Christoph supported by SNSF Ambizione Grant No. 216071. Research of Lukas Michel supported by an ETH Zürich Postdoctoral Fellowship. Parts of this research were carried out while Lukas Michel was at the University of Oxford.} \and Nemanja Dragani\'{c}\affiliation{Department of Computing Sciences, Bocconi University, Milan, Italy (\textsf{\href{mailto:nemanja.draganic@unibocconi.it}{nemanja.draganic@unibocconi.it}}). Research supported in part by the Javotte Bocconi Institute. Parts of this research were carried out while the author was at the University of Oxford, supported by SNSF grant 217926, and at EPF Lausanne, supported by SNSF grant P5R5-2\_239120.} \and Ant\'{o}nio Gir\~{a}o\affiliation{Mathematical Institute, University of Oxford, United Kingdom (\textsf{\{\href{mailto:antonio.girao@maths.ox.ac.uk}{antonio.girao},\href{mailto:eoin.hurley@maths.ox.ac.uk}{eoin.hurley},\href{mailto:alp.muyesser@maths.ox.ac.uk}{alp.muyesser}\}\allowbreak@maths.ox.ac.uk}). Research of Eoin Hurley supported by ERC Advanced Grant 883810.} \and Eoin Hurley\affiliationmark[3] \and Lukas Michel\affiliationmark[1] \and Alp M\"{u}yesser\affiliationmark[3]}

\date{29 September 2026}

\begin{document}

\maketitle

\begin{abstract}
  It is a classical result that a random permutation of $n$ elements has, on average, about $\log n$ cycles. We generalise this fact to all directed $d$-regular graphs on $n$ vertices by showing that, on average, a random cycle-factor of such a graph has $\cO((n\log d)/d)$ cycles. This is tight up to the constant factor and improves the best previous bound of the form $\cO({n/\!\sqrt{\log d}})$ due to Vishnoi. Our results also yield randomised polynomial-time algorithms for finding such a cycle-factor and for finding a tour of length $(1+\cO((\log d)/d)) \cdot n$ if the graph is connected. This makes progress on a conjecture of Magnant and Martin and on a problem studied by Vishnoi and by Feige, Ravi, and Singh. Our proof uses the language of entropy to exploit the fact that the upper and lower bounds on the number of perfect matchings in regular bipartite graphs are extremely close.
\end{abstract}

\section{Introduction}

The probability that a random permutation has a fixed point was considered as early as 1704 in Pierre-R\'emond de Montmort's \emph{Essai d’Analyse sur les Jeux de Hasard}. Random permutations have since become one of the most classical topics in discrete probability. Cauchy was the first to study permutations for their own sake. In the 1840s, he derived a formula for the number of permutations of $n$ elements with a given cycle structure \cite{cauchy1840exercices}. This formula implies that, on average, a random permutation contains $\sum_{k=1}^n 1/k \simeq \log n$ cycles. The latter result is now a standard exercise in discrete probability.

In this paper, we generalise this result to all directed $d$-regular graphs\footnote{A directed graph can have loops and directed cycles of length $2$, but no parallel edges. A directed graph is \emph{$d$-regular} if the in-degree and out-degree of every vertex is exactly $d$.} on $n$ vertices.  A \emph{cycle-factor} of such a graph is a partition of its vertices into directed cycles. Equivalently, this corresponds to a permutation of the vertices such that every vertex is mapped to one of its $d$ out-neighbours. In particular, a random permutation of $n$ elements is the same as a random cycle-factor of the complete directed graph on $n$ vertices with a loop at every vertex.

It follows from Hall's theorem that every directed $d$-regular graph has at least one cycle-factor. How many cycles are there in a typical cycle-factor of such a graph? By the discussion above, a random cycle-factor of the disjoint union of $n/d$ complete directed graphs on $d$ vertices has, on average, about $(n \log d) / d$ cycles. We prove that, up to constant factors, this is an upper bound on the expected number of cycles in a typical cycle-factor for any directed $d$-regular graph.

\begin{theorem}\label{thm:smallcyclefactordirected}
   The expected number of cycles in a uniformly random cycle-factor of a directed $d$-regular graph on $n$ vertices is $\cO((n\log d)/d)$.
\end{theorem}

This recovers the aforementioned fact that, on average, a uniformly random permutation of $n$ elements has $\cO(\log n)$ cycles. The constant factor in \cref{thm:smallcyclefactordirected} can be taken to be $4$. In addition, we can show that the distribution of the number of cycles in a uniformly random cycle-factor has a one-sided exponential tail behaviour.

\begin{theorem}\label{thm:numbercyclesconcentration}
    Let $C$ be a uniformly random cycle-factor of a directed $d$-regular graph $G$ on $n$ vertices. Then, for any real number $t$ we have
    \[
        \pr\left(C \text{ has at least } 4 \cdot \frac{n}{d} (\log d + 1) + t \text{ cycles}\right) \le 2^{-t}.
    \]
\end{theorem}

Note that \cref{thm:smallcyclefactordirected} also applies to undirected $d$-regular graphs if we direct each edge in both directions\footnote{A cycle-factor of an undirected graph can contain individual edges, which we consider to be cycles of length $2$.}. Previously, the best upper bound known on the number of cycles in a typical cycle-factor of a $d$-regular graph was $\cO({n/\!\sqrt{\log d}})$ due to Vishnoi \cite{vishnoi2012permanent}.

While \cref{thm:smallcyclefactordirected,thm:numbercyclesconcentration} only prove the existence of cycle-factors with few cycles, we can also find such cycle-factors efficiently. 
This is a consequence of the fact that we can sample almost-uniform random perfect matchings of bipartite graphs in polynomial time \cite{jerrum2004polynomial}.

\begin{corollary}\label{cor:algorithmsmallcyclefactordirected}
    There exists a randomised polynomial-time algorithm that computes, with high probability, a cycle-factor of a directed $d$-regular graph on $n$ vertices with at most $\cO((n\log d)/d)$ cycles.
\end{corollary}

If we remove one edge from each cycle in a cycle-factor, we obtain a \emph{path-factor}, which is a partition of the vertices of a graph into paths. Magnant and Martin \cite{magnant2009note} conjectured that every $d$-regular graph on $n$ vertices has a path-factor with at most $n/(d+1)$ paths. This would be tight as the disjoint union of $n/(d+1)$ cliques on $d+1$ vertices requires at least $n/(d+1)$ paths in any path-factor. The conjecture of Magnant and Martin is only known to be true for $d \le 6$ \cite{magnant2009note,feige2022path}, $d = \Omega(n)$ \cite{gruslys2021cycle}, and for partitioning almost all vertices of a graph into paths \cite{montgomery2024approximate,letzter2025nearly}. Feige, Ravi, and Singh \cite{feige2014short} showed that every $d$-regular graph on $n$ vertices has a path-factor with $\cO(n/\!\sqrt{d})$ paths, which was the best previous upper bound known. Using \cref{cor:algorithmsmallcyclefactordirected}, we get the following improvement.

\begin{corollary}\label{thm:smallpathfactor}
    There exists a randomised polynomial-time algorithm that computes, with high probability, a path-factor of a $d$-regular graph on $n$ vertices with at most $\cO((n\log d)/d)$ paths.
\end{corollary}

The conjecture of Magnant and Martin is also intimately related to the celebrated Linear Arboricity Conjecture of Akiyama, Exoo, and Harary \cite{akiyama1980covering} which states that the edges of every graph with maximum degree $\Delta$ can be partitioned into at most $\lceil (\Delta+1)/2 \rceil$ path-factors. The best bound known for the Linear Arboricity Conjecture is due to Lang and Postle \cite{lang2023improved} and shows that $\Delta/2 + \cO(\sqrt{\Delta} \log^4 \Delta)$ path-factors suffice. This bound cannot be improved to $\Delta/2 + o(\sqrt{\Delta})$ without being able to show that every $\Delta$-regular graph has a path-factor with at most $o(n/\!\sqrt{\Delta})$ paths. \Cref{thm:smallpathfactor} shows that such path-factors exist.
 
Finally, we can use our results to find short tours\footnote{A \emph{tour} of a graph is a connected walk that visits every vertex and starts and ends at the same vertex.} of connected $d$-regular graphs. The problem of finding such tours was introduced by Vishnoi \cite{vishnoi2012permanent}, motivated by the study of approximation algorithms for the travelling salesman problem. Feige, Ravi, and Singh \cite{feige2014short} showed that there exist tours of length $(1+\cO(1/\!\sqrt{d})) \cdot n$ and provided an example in which every tour has length at least $(1+\Omega(1/d)) \cdot n$. Since a cycle-factor with $c$ cycles in a connected graph can be efficiently transformed into a tour of length $n + 2 c$, we obtain the following.

\begin{corollary}\label{cor:smalltour}
    There exists a randomised polynomial-time algorithm that computes, with high probability, a tour of a connected $d$-regular graph on $n$ vertices of length at most $(1+\cO((\log d)/d)) \cdot n$.
\end{corollary}

The rest of the paper is structured as follows. In \cref{sec:expcyclesincyclefactor}, we prove our main results, \cref{thm:smallcyclefactordirected,thm:numbercyclesconcentration,cor:algorithmsmallcyclefactordirected}. This section also includes an overview of our proof strategy. Since the proofs of \cref{thm:smallcyclefactordirected,thm:numbercyclesconcentration} use the language of entropy, we first provide a short introduction to entropy in \cref{sec:prelims}. We also use this section to introduce a small amount of notation. We finish in \cref{sec:openproblems} with some open problems.

\textbf{Note.} \emph{In subsequent work \cite{christoph2025lineararboricity}, we used a completely different approach to prove that every $d$-regular graph on $n$ vertices has a path-factor with $\cO(n/d)$ paths and a tour of length $(1 + \cO(1/d)) \cdot n$ if the graph is connected. We also used this approach to make progress on the Linear Arboricity Conjecture. However, this approach does not provide efficient algorithms and does not apply to directed graphs or cycle-factors.}

\section{Entropy}\label{sec:prelims}

In this section, we recall some basic definitions and properties of Shannon entropy that will be essential for our proofs. Suppose that $X$ is a random variable that takes values in a finite set $\X$. Then, the \emph{Shannon entropy} of $X$ is defined as
\[
    \ent(X) \coloneqq - \sum_{x \in \X} \pr(X = x) \logt(\pr(X = x)).
\]
Here, we use the convention $p \logt p = 0$ for $p = 0$. One basic property of entropy is that $\ent(X) \le \logt \abs{\X}$ with equality if and only if $X$ is distributed uniformly at random in $\X$. For $0 \le p \le 1$, we write $\ent(p) \coloneqq - p \logt p - (1-p) \logt(1-p)$ for the entropy of a Bernoulli random variable with probability $p$.

The definition of entropy easily extends to conditional probability. If $E$ is some event, then the entropy of $X$ conditioned on $E$ occurring is
\[
    \ent(X \mid E) \coloneqq - \sum_{x \in \X} \pr(X = x \mid E) \logt(\pr(X = x \mid E)).
\]
If $Y$ is another random variable that takes values in a finite set $\Y$, then the \emph{conditional entropy} of $X$ given $Y$ is defined as
\[
    \ent(X \mid Y) \coloneqq \sum_{y \in \Y} \pr(Y = y) \cdot \ent(X \mid Y = y).
\]
We remark that $\ent(X \mid Y) = 0$ if and only if $Y$ completely determines $X$.

The only fact about conditional entropy that we will need is the \emph{chain rule}, which says that if $X_1, \dots, X_n$ are random variables, then
\[
    \ent(X_1, \dots, X_n) = \ent(X_1) + \ent(X_2 \mid X_1) + \dots + \ent(X_n \mid X_1, \dots, X_{n-1})
\]
where the tuple $X_1, \dots, X_i$ is viewed as a single random variable whose distribution is the joint distribution of $X_1, \dots, X_i$.

\textbf{Notation.} Let $[n] \coloneqq \{1, \dots, n\}$. A permutation of $[n]$ is a bijection $\sigma: [n] \to [n]$. We denote the family of all permutations of $[n]$ by $S_n$. If $\sigma \in S_n$ and $\cI \subseteq [n]$, we write $\sigma(\cI) \coloneqq \{\sigma(i) : i \in \cI\}$. A cycle of $\sigma \in S_n$ is a minimal non-empty subset $\cI \subseteq [n]$ such that $\sigma(\cI) = \cI$. We denote the number of cycles of $\sigma$ by $\abs{\sigma}$.

\section{Expected number of cycles in a random cycle-factor}
\label{sec:expcyclesincyclefactor}

Suppose that $G$ is a directed $d$-regular graph on $n$ vertices. We wish to show that the expected number of cycles of a uniformly random cycle-factor of $G$ is $\cO((n\log d)/d)$. To achieve this, we will reveal the edges of the random cycle-factor one by one in a random order and bound the probability that each new revealed edge closes a cycle. We do this by going through the vertices in a random order and revealing their out-neighbours in the cycle-factor. If the revealed out-neighbour of a vertex were always uniformly distributed among all remaining available out-neighbours, then we could show the desired bound. Indeed, since the edges are revealed in a random order, it is possible to show that the number $t$ of remaining out-neighbours of a vertex is uniformly distributed in $[d]$, and so the probability that each new revealed edge closes a cycle would be at most
\[
    \frac{1}{d} \sum_{t = 1}^d \frac{1}{t} = \cO\left(\frac{\log d}{d}\right).
\]

To handle the general case where out-neighbours are not uniformly distributed, we connect our process to Radhakrishnan's entropy proof \cite{radhakrishnan1997entropy} of the Br\`{e}gman--Minc inequality \cite{bregman1973some}. To that end, construct a $d$-regular bipartite graph $H$ with bipartition $V(H) = U \cup V$ where $U$ and $V$ are both copies of $V(G)$, and for each directed edge $(u,v)$ of $G$ we add an edge between $u \in U$ and $v \in V$ to $H$. It is easy to see that cycle-factors of $G$ correspond to perfect matchings of $H$\footnote{This observation was also used by Vishnoi in \cite{vishnoi2012permanent} to show that a typical cycle-factor in a $d$-regular graph has $\cO(n/\sqrt{\log d})$ cycles.}. The Br\`{e}gman--Minc inequality gives the following upper bound on the number of perfect matchings in regular bipartite graphs.

\begin{theorem}\label{thm:bregmanmincbipperfmat}
    Every $d$-regular bipartite graph on $2 n$ vertices has at most $(d!)^{n/d}$ perfect matchings.
\end{theorem}

At a high level, Radhakrishnan's proof of \Cref{thm:bregmanmincbipperfmat} picks a uniformly random perfect matching and shows that its entropy is at most $\logt((d!)^{n/d})$. The idea here is to reveal the edges of the random matching one by one in a random order\footnote{To the best of our knowledge, the random order was first used by Alon and Spencer \cite{alon2016probabilistic} to prove the Br\`{e}gman--Minc inequality.} and to bound the amount of entropy added by each new revealed edge. Note that the amount of entropy added is maximised if each new revealed edge is uniformly distributed among its remaining options. This can be used to prove \cref{thm:bregmanmincbipperfmat}.

We pursue a very similar strategy to analyse random cycle-factors of $G$. Conveniently, revealing an edge of a random perfect matching of the auxiliary bipartite graph $H$ corresponds precisely to revealing an edge of the random cycle-factor of $G$. As mentioned before, we wish to understand the case where a revealed edge of the cycle-factor is not uniformly distributed among its remaining options. If this happens, the amount of entropy added for the corresponding edge in $H$ is lower than the entropy of the uniform distribution. This, in turn, translates into a better upper bound on the number of perfect matchings of $H$. However, this upper bound cannot drop too low. This is because the Van der Waerden inequality, proved independently by Egorychev \cite{egorychev1981solution} and Falikman \cite{falikman1981proof}, gives the following lower bound on the number of perfect matchings in regular bipartite graphs.

\begin{theorem}\label{thm:vanderwaerdenbipperfmat}
    Every $d$-regular bipartite graph on $2 n$ vertices has at least $n! \cdot d^n / n^n$ perfect matchings.
\end{theorem}

This lower bound allows us to quantify how much entropy can be lost while revealing the edges of the random cycle-factor of $G$. We will see that the total loss in entropy is at most $\cO((n\log d)/d)$ compared to what the entropy would have been had the edges been uniformly distributed throughout the process. The key fact that ties everything together is that any loss in entropy increases the probability of closing a cycle at most by an additive linear term. Together, this implies \cref{thm:smallcyclefactordirected}.

We start by proving the following simple lemma that formalises the relationship between a loss in entropy and an increase in the probability of a random variable attaining a specific value. In our case, this will correspond to the probability of closing a cycle.

\begin{lemma}\label{lem:entlossskew}
    Let $X$ be a random variable taking values in a finite set $\X$, and let $\ell \coloneqq \logt \abs{\X} - \ent(X)$. Then, for all $x \in \X$ we have
    \[
        \pr(X = x) \le \frac{2}{\abs{\X}} + \ell.
    \]
\end{lemma}

\begin{proof}
    Let $x \in \X$ and let $p \coloneqq \pr(X = x)$. Consider the indicator random variable $\ind_{X = x}$ of the event that $X = x$. By the chain rule for entropy,
    \[
        \ent(X) = \ent(\ind_{X = x}) + \ent(X \mid \ind_{X = x}) = \ent(p) + p \cdot \ent(X \mid X = x) + (1-p) \cdot \ent(X \mid X \neq x).
    \]
    Clearly, $\ent(X \mid X = x) = 0$. Also, if $X \neq x$, then $X$ takes at most $\abs{\X}-1$ distinct values, and so $\ent(X \mid X \neq x) \le \logt(\abs{\X}-1)$. This implies that
    \[
        \ent(X) \le \ent(p) + (1-p) \logt(\abs{\X}-1).
    \]
    Let $h(q) \coloneqq \ent(q) + (1-q) \logt(\abs{\X}-1)$. Note that $h(q) = \ent(Y)$ for the random variable $Y$ that satisfies $\pr(Y = x) = q$ and $\pr(Y = y) = (1-q)/(\abs{\X}-1)$ for all $y \in \X \setminus \{x\}$, and so $h(q) \le \logt \abs{\X}$. Moreover,
    \[
        h'(q) = - \logt q + \logt(1-q) - \logt(\abs{\X}-1) = \logt\left(\frac{1-q}{q (\abs{\X}-1)}\right).
    \]
    Since this is decreasing in $q$, we get for $q \ge 2/\abs{\X}$ that
    \[
        h'(q) \le h'\left(\frac{2}{\abs{\X}}\right) = \logt\left(\frac{\abs{\X}-2}{2 (\abs{\X}-1)}\right) \le \logt\left(\frac{1}{2}\right) = -1.
    \]
    So, for $q \ge 2/\abs{\X}$, we have $h(q) \le h(2/\abs{\X}) - (q - 2/\abs{\X}) \le \logt \abs{\X} - (q - 2/\abs{\X})$. Since we know that $h(p) \ge \ent(X) = \logt \abs{\X} - \ell$, it follows that $p \le 2/\abs{\X} + \ell$.
\end{proof}

We can now bound the expected number of cycles of a random cycle-factor of $G$ using the strategy outlined earlier in this section. The first part of the proof closely follows Radhakrishnan's proof of the Br\`{e}gman--Minc inequality (see also \cite[Chapter 37]{aigner1999proofs} for a nice exposition of this proof), with the twist being that we keep track of the total loss in entropy. This is used in the second part to prove an upper bound on the expected number of cycles, by passing through \cref{lem:entlossskew}. Note that instead of only considering uniformly random cycle-factors of $G$, we prove a slightly more general result that applies to any random cycle-factor of $G$. Later, this will allow us to also prove a one-sided concentration bound on the number of cycles in a uniformly random cycle-factor.

\begin{proposition}\label{prop:boundexpectednumberofcycles}
    Let $C$ be a random cycle-factor of a directed $d$-regular graph $G$ on $n$ vertices. Then, the expected number of cycles of $C$ is at most
    \[
        2 \cdot \frac{n}{d} (\log d + 1) + \frac{n}{d} \logt(d!) - \ent(C).
    \]
\end{proposition}

\begin{proof}
    Identify the vertices of $G$ with $[n]$. We will view $C$ as the random permutation $\sigma$ of $[n]$ where $\sigma(i)$ is the out-neighbour of $i$ in $C$ for all $i \in [n]$. In particular, the number of cycles of $C$ is exactly $\abs{\sigma}$, and $\ent(C) = \ent(\sigma)$. Also, let $\C$ be the collection of all permutations $\sigma'$ of $[n]$ with $(i,\sigma'(i)) \in E(G)$ for all $i \in [n]$, so $\C$ corresponds to the collection of all cycle-factors of $G$.

    For every permutation $\tau \in S_n$, the chain rule for entropy implies that
    \[
        \ent(\sigma) = \sum_{j=1}^n \ent(\sigma(\tau(j)) \mid \sigma(\tau(1)), \dots, \sigma(\tau(j-1))).
    \]
    For $i \in [n]$, let $\tau_{<i} \coloneqq \tau(1), \dots, \tau(j-1)$, where $j \coloneqq \tau^{-1}(i)$. We will use the notation $\sigma(\tau_{<i}) \coloneqq \sigma(\tau(1)), \dots, \sigma(\tau(j-1))$. If we reorder the above sum, we get
    \[
        \ent(\sigma) = \sum_{i=1}^n \ent(\sigma(i) \mid \sigma(\tau_{<i})).
    \]
    Expanding the definition of conditional entropy, this yields
    \[
        \ent(\sigma) = \sum_{i=1}^n \sum_\pi \pr(\sigma(\tau_{<i}) = \pi) \cdot \ent(\sigma(i) \mid \sigma(\tau_{<i}) = \pi),
    \]
    where $\pi$ ranges over all possible values of $\sigma(\tau_{<i})$. Observe that $\pr(\sigma(\tau_{<i}) = \pi)$ is exactly the sum of the probabilities $\pr(\sigma = \sigma')$ over all $\sigma' \in \C$ with $\sigma'(\tau_{<i}) = \pi$. Therefore,
    \[
        \ent(\sigma) = \sum_{i=1}^n \sum_{\sigma' \in \C} \pr(\sigma = \sigma') \cdot \ent(\sigma(i) \mid \sigma(\tau_{<i}) = \sigma'(\tau_{<i})).
    \]
    For all $i \in [n]$, $\sigma' \in \C$, and $\tau \in S_n$, let $a(i,\sigma',\tau) \coloneqq \abs{N_G^+(i) \setminus \sigma'(\tau_{<i})}$ be the number of \emph{available out-neighbours} of $i$ in $G$, that is, the number of out-neighbours of $i$ in $G$ that are not contained in $\sigma'(\tau_{<i})$. If $\sigma(\tau_{<i}) = \sigma'(\tau_{<i})$, then $\sigma(i)$ must be one of these available out-neighbours, and so $\ent(\sigma(i) \mid \sigma(\tau_{<i}) = \sigma'(\tau_{<i})) \le \logt a(i,\sigma',\tau)$. Denote the loss in entropy in this inequality by $\ell(i,\sigma',\tau) \coloneqq \logt a(i,\sigma',\tau) - \ent(\sigma(i) \mid \sigma(\tau_{<i}) = \sigma'(\tau_{<i}))$. Then,
    \[
        \ent(\sigma) = \sum_{i=1}^n \sum_{\sigma' \in \C} \pr(\sigma = \sigma') \cdot (\logt a(i,\sigma',\tau) - \ell(i,\sigma',\tau)).
    \]
    Since this holds for all $\tau \in S_n$, we may pick $\tau \in S_n$ uniformly at random. This yields
    \begin{align*}
        \ent(\sigma) & = \ev_\tau\left(\sum_{i=1}^n \sum_{\sigma' \in \C} \pr(\sigma = \sigma') \cdot (\logt a(i,\sigma',\tau) - \ell(i,\sigma',\tau))\right) \\
        & = \sum_{i=1}^n \sum_{\sigma' \in \C} \pr(\sigma = \sigma') \cdot \ev_\tau(\logt a(i,\sigma',\tau)) - \cL
    \end{align*}
    where
    \[
        \cL \coloneqq \sum_{i=1}^n \sum_{\sigma' \in \C} \pr(\sigma = \sigma') \cdot \ev_\tau(\ell(i,\sigma',\tau))
    \]
    denotes the loss in entropy compared to a uniform distribution.
    
    For a fixed $i$ and $\sigma'$, note that $a(i,\sigma',\tau)$ is the number of out-neighbours $j \in N_G^+(i)$ of $i$ in $G$ such that $(\sigma')^{-1}(j)$ does not appear before $i$ in $\tau$. Since $\tau$ is a uniformly random permutation, this number is distributed uniformly at random in $[d]$, and so it follows that
    \[
        \ent(\sigma) = \sum_{i=1}^n \sum_{\sigma' \in \C} \pr(\sigma = \sigma') \cdot \left(\frac{1}{d} \sum_{t=1}^d \logt t\right) - \cL = \frac{n}{d} \logt(d!) - \cL.
    \]
    In particular, we get that $\cL = n \logt(d!) / d - \ent(\sigma) = n \logt(d!) / d - \ent(C)$.

    We now bound the number of cycles $\abs{\sigma}$ of $\sigma$ using a very similar calculation. For every permutation $\tau \in S_n$, we say that $\sigma(i)$ closes a cycle if $\tau_{<i}$ contains all other vertices that are on the same cycle as $i$ in $\sigma$. That is, adding $\sigma(i)$ to $\sigma(\tau_{<i})$ creates one additional cycle. Note that $\abs{\sigma}$ is equal to the number of vertices $i \in [n]$ for which $\sigma(i)$ closes a cycle, and so
    \begin{align*}
        \ev_\sigma(\abs{\sigma}) & = \sum_{i = 1}^n \pr(\sigma(i) \text{ closes a cycle}) \\
        & = \sum_{i=1}^n \sum_\pi \pr(\sigma(\tau_{<i}) = \pi) \cdot \pr(\sigma(i) \text{ closes a cycle} \mid \sigma(\tau_{<i}) = \pi) \\
        & = \sum_{i=1}^n \sum_{\sigma' \in \C} \pr(\sigma = \sigma') \cdot \pr(\sigma(i) \text{ closes a cycle} \mid \sigma(\tau_{<i}) = \sigma'(\tau_{<i})).
    \end{align*}
    If $\sigma(\tau_{<i}) = \sigma'(\tau_{<i})$, either $\sigma(i)$ cannot close a cycle, or $\sigma(i)$ needs to be one specific available out-neighbour $j \in N_G^+(i) \setminus \sigma'(\tau_{<i})$ in order to close a cycle. So, by \cref{lem:entlossskew} we know that
    \begin{align*}
        \pr(\sigma(i) \text{ closes a cycle} \mid \sigma(\tau_{<i}) = \sigma'(\tau_{<i})) & \le \max_{j \in N_G^+(i)} \pr(\sigma(i) = j \mid \sigma(\tau_{<i}) = \sigma'(\tau_{<i})) \\
        & \le \frac{2}{a(i,\sigma',\tau)} + \ell(i,\sigma',\tau).
    \end{align*}
    This implies that
    \[
        \ev_\sigma(\abs{\sigma}) \le \sum_{i=1}^n \sum_{\sigma' \in \C} \pr(\sigma = \sigma') \cdot \left(\frac{2}{a(i,\sigma',\tau)} + \ell(i,\sigma',\tau)\right).
    \]
    Since this holds for all $\tau \in S_n$, we may again pick $\tau \in S_n$ uniformly at random. As noted before, $a(i,\sigma',\tau)$ is then distributed uniformly at random in $[d]$, and so
    \begin{align*}
        \ev_\sigma(\abs{\sigma}) & \le \ev_\tau\left(\sum_{i=1}^n \sum_{\sigma' \in \C} \pr(\sigma = \sigma') \cdot \left(\frac{2}{a(i,\sigma',\tau)} + \ell(i,\sigma',\tau)\right)\right) \\
        & = \sum_{i=1}^n \sum_{\sigma' \in \C} \pr(\sigma = \sigma') \cdot \ev_\tau\left(\frac{2}{a(i,\sigma',\tau)}\right) + \cL \\
        & = \sum_{i=1}^n \sum_{\sigma' \in \C} \pr(\sigma = \sigma') \cdot \left(\frac{1}{d} \sum_{t = 1}^d \frac{2}{t}\right) + \cL \\
        & \le 2 \cdot \frac{n}{d} (\log d + 1) + \cL = 2 \cdot \frac{n}{d} (\log d + 1) + \frac{n}{d} \logt(d!) - \ent(C). \qedhere
    \end{align*}
\end{proof}

It is now easy to deduce the desired upper bound on the expected number of cycles of a uniformly random cycle-factor of $G$ by using the Van der Waerden inequality to get a lower bound on the entropy of such a cycle-factor.

\begin{proof}[Proof of \cref{thm:smallcyclefactordirected}]
    Let $C$ be a uniformly random cycle-factor of $G$, and let $\C$ be the collection of all cycle-factors of $G$. Since cycle-factors of $G$ correspond to perfect matchings in the auxiliary bipartite graph $H$ from before, \cref{thm:vanderwaerdenbipperfmat} implies that $\abs{\C} \ge n! \cdot d^n / n^n$. As $n! \ge (n/e)^n$, this gives $\abs{\C} \ge (d/e)^n$ and therefore $\ent(C) = \logt \abs{\C} \ge n \logt(d/e)$. Since $d! \le e d (d/e)^d$, this implies that
    \[
        \frac{n}{d} \logt(d!) - \ent(C) \le n \logt\left(\frac{d}{e}\right) + \frac{n}{d} \logt(ed) - n \logt\left(\frac{d}{e}\right) = \frac{n}{d} \logt(ed).
    \]
    Therefore, by \cref{prop:boundexpectednumberofcycles}, the expected number of cycles of $C$ is at most
    \[
        2 \cdot \frac{n}{d} (\log d + 1) + \frac{n}{d} \logt(ed) \le 4 \cdot \frac{n}{d} (\log d + 1). \qedhere
    \]
\end{proof}

In addition, we can now also prove that the distribution of the number of cycles of a uniformly random cycle-factor of $G$ has a one-sided exponential tail behaviour.

\begin{proof}[Proof of \cref{thm:numbercyclesconcentration}]
    Let $\C$ be the collection of all cycle-factors of $G$, and let $\C' \subseteq \C$ denote the subset of all cycle-factors of $G$ with at least $k \coloneqq 4 n (\log d + 1) / d + t$ cycles. Observe that the probability of the event that $C$ has at least $k$ cycles is exactly $\abs{\C'}/\abs{\C}$.
    
    Let $C'$ be a cycle-factor chosen uniformly at random from $\C'$. Since the expected number of cycles of $C'$ is at least $k$, \cref{prop:boundexpectednumberofcycles} implies that
    \[
        k \le 2 \cdot \frac{n}{d} (\log d + 1) + \frac{n}{d} \logt(d!) - \ent(C').
    \]
    Using calculations similar to those in the proof of \cref{thm:smallcyclefactordirected}, we get
    \[
        \ent(C') \le n \logt\left(\frac{d}{e}\right) + 4 \cdot \frac{n}{d} (\log d + 1) - k \le \logt \abs{\C} - t.
    \]
    Since $\ent(C') = \logt \abs{\C'}$, it follows that $\pr(C \text{ has at least } k \text{ cycles}) = \abs{\C'}/\abs{\C} \le 2^{-t}$.
\end{proof}

We finish by showing that there is a randomised polynomial-time algorithm that finds a cycle-factor of $G$ with at most $\cO((n\log d)/d)$ cycles with high probability. For this, we will use the fact that there exists an algorithm that samples an almost-uniform\footnote{If $X$ is a random variable that takes values in a finite set $\X$, then $X$ is \emph{almost-uniform} with error $\delta$ if the total variation distance between $X$ and the uniform distribution is at most $\delta$, that is $\sum_{x \in \X} \abs{1/\abs{\X} - \pr(X = x)}/2 \le \delta$.} random perfect matching in bipartite graphs, which was proved by Jerrum, Sinclair, and Vigoda \cite{jerrum2004polynomial}.

\begin{theorem}\label{thm:algorithmalmostuniformperfectmatching}
    There exists a randomised polynomial-time algorithm that samples an almost-uniform random perfect matching of a bipartite graph.
\end{theorem}

Using the fact that cycle-factors correspond to perfect matchings in the bipartite graph $H$ constructed earlier, \cref{cor:algorithmsmallcyclefactordirected} is now an easy consequence.

\begin{proof}[Proof of \cref{cor:algorithmsmallcyclefactordirected}]
    Recall that cycle-factors of $G$ correspond to perfect matchings in the bipartite graph $H$ constructed earlier. By \cref{thm:algorithmalmostuniformperfectmatching} we can sample an almost-uniform random perfect matching of $H$ in polynomial time, and so we can sample an almost-uniform random cycle-factor $\sigma$ of $G$ in polynomial time.

    By \cref{thm:smallcyclefactordirected} and Markov's inequality, we know that a uniformly random cycle-factor of $G$ has at most $\cO((n\log d)/d)$ cycles with probability at least $1/2$. Since $\sigma$ is almost-uniform with some error $\delta$, this implies that $\sigma$ has at most $\cO((n\log d)/d)$ cycles with probability at least $1/2 - \delta$. If we now sample $\sigma$ independently at least $\Omega(\log n)$ times and return the cycle-factor with the least number of cycles, then the returned cycle-factor has at most $\cO((n\log d)/d)$ cycles with high probability.
\end{proof}

\section{Open problems}\label{sec:openproblems}

We believe that \cref{thm:smallcyclefactordirected} should hold with a constant factor of $1 + o(1)$. Originally, we conjectured that a stability version of this result should hold in the sense that the expected number of cycles in a uniformly random cycle-factor is maximised by the disjoint union of $n/d$ complete directed graphs on $d$ vertices with a loop at every vertex. However, this stronger conjecture was recently disproved by Gajjala \cite{gajjala2026counterexamples}.

\begin{conjecture}
    The expected number of cycles in a uniformly random cycle-factor of a directed $d$-regular graph on $n$ vertices is at most $(1 + o(1)) (n \log d) / d$ as $d \to \infty$.
\end{conjecture}

It is natural to ask what other properties of random permutations might generalise to directed regular graphs. As an example, de Montmort's problem about the probability of a fixed point does not generalise. Indeed, in the complete directed bipartite graph between a set $A$ of $d$ vertices and a set $B$ of $d-1$ vertices with a loop at every vertex of $A$, every cycle-factor has exactly one fixed point, and so the probability that a random cycle-factor has a fixed point is $1$. In contrast, the probability that a random permutation has a fixed point is approximately $1-1/e$. However, it remains possible that the expected number of fixed points generalises.

\begin{problem}\label{prob:fixedpointscyclefactor}
    Is the expected number of fixed points in a uniformly random cycle-factor of a directed $d$-regular graph on $n$ vertices at most $\cO(n/d)$?
\end{problem}

Another property of random permutations is that every element is equally likely to be mapped to every other element. The corresponding generalisation to directed regular graphs would have answered \cref{prob:fixedpointscyclefactor} affirmatively. Unfortunately, this generalisation is completely false since there are directed $d$-regular graphs with an edge that is contained in a uniformly random cycle-factor with probability $1-o(1)$. This follows from a construction in \cite[Theorem 1.7]{abert2016matchings}. We note that this generalisation turns out to be correct in the special case of dense and robustly expanding regular bipartite graphs \cite[Theorem 3.3]{granet2024random}.

In general, it would be interesting to identify further statistics of the cycle structure of random permutations that admit analogues for random cycle-factors of directed regular graphs.

\textbf{Acknowledgements.} We would like to thank Thomas Karam for an informative lecture on Radhakrishnan's entropy proof of the Br\`{e}gman--Minc inequality.

\bibliography{bib}

@article{christoph2025lineararboricity,
    author = {Christoph, Micha and Dragani{\'{c}}, Nemanja and Gir{\~{a}}o, Ant{\'{o}}nio and Hurley, Eoin and Michel, Lukas and M{\"{u}}yesser, Alp},
    title = {New bounds for linear arboricity and related problems},
    year = {2025},
    journal = {Journal of the European Mathematical Society,},
    arxiv = {2507.20500},
    pages = {to appear}
}

@article{abert2016matchings,
    author = {Ab{\'e}rt, Mikl{\'o}s and Csikv{\'a}ri, P{\'e}ter and Frenkel, P{\'e}ter E. and Kun, G{\'a}bor},
    title = {Matchings in {B}enjamini--{S}chramm convergent graph sequences},
    journal = {Transactions of the American Mathematical Society},
    volume = {368},
    year = {2016},
    number = {6},
    pages = {4197--4218},
    doi = {10.1090/tran/6464}
}

@book{cauchy1840exercices,
  title={Exercices d'analyse et de physique mathematique: 1},
  author={Cauchy, Augustin Louis},
  volume={1},
  year={1840},
  publisher={Bachelier, imprimeur-libraire}
}

@article{granet2024random,
    author = {Granet, Bertille and Joos, Felix},
    title = {Random perfect matchings in regular graphs},
    journal = {Random Structures \& Algorithms},
    volume = {64},
    year = {2024},
    number = {1},
    pages = {3--14},
    doi = {10.1002/rsa.21172}
}

@book{alon2016probabilistic,
    author = {Alon, Noga and Spencer, Joel H.},
    title = {The probabilistic method},
    notusedseries = {Wiley Series in Discrete Mathematics and Optimization},
    edition = {Fourth},
    publisher = {John Wiley \& Sons},
    notusedpublisher = {John Wiley \& Sons, Inc., Hoboken, NJ},
    year = {2016},
    pages = {xiv+375}
}

@inproceedings{vishnoi2012permanent,
    author = {Vishnoi, Nisheeth K.},
    title = {A permanent approach to the traveling salesman problem},
    booktitle = {2012 {IEEE} 53rd {A}nnual {S}ymposium on {F}oundations of {C}omputer {S}cience},
    pages = {76--80},
    publisher = {IEEE Computer Society},
    year = {2012},
    doi = {10.1109/FOCS.2012.81}
}

@article{magnant2009note,
    author = {Magnant, Colton and Martin, Daniel M.},
    title = {A note on the path cover number of regular graphs},
    journal = {The Australasian Journal of Combinatorics},
    volume = {43},
    year = {2009},
    pages = {211--217},
    url = {https://ajc.maths.uq.edu.au/pdf/43/ajc_v43_p211.pdf}
}

@article{gruslys2021cycle,
    author = {Gruslys, Vytautas and Letzter, Shoham},
    title = {Cycle partitions of regular graphs},
    journal = {Combinatorics, Probability and Computing},
    volume = {30},
    year = {2021},
    number = {4},
    pages = {526--549},
    doi = {10.1017/s0963548320000553}
}

@article{montgomery2024approximate,
    author = {Montgomery, Richard and M{\"u}yesser, Alp and Pokrovskiy, Alexey and Sudakov, Benny},
    title = {Approximate path decompositions of regular graphs},
    journal = {Journal of the London Mathematical Society},
    volume = {112},
    year = {2025},
    number = {2},
    pages = {Paper No. e70269, 40},
    doi = {10.1112/jlms.70269}
}

@article{letzter2025nearly,
    author = {Letzter, Shoham and Methuku, Abhishek and Sudakov, Benny},
    title = {Nearly {H}amilton cycles in sublinear expanders and applications},
    journal = {Journal of the London Mathematical Society},
    volume = {113},
    year = {2026},
    number = {2},
    pages = {Paper No. e70452, 42},
    doi = {10.1112/jlms.70452}
}

@article{feige2022path,
    author = {Feige, Uriel and Fuchs, Ella},
    title = {On the path partition number of 6-regular graphs},
    journal = {Journal of Graph Theory},
    volume = {101},
    year = {2022},
    number = {3},
    pages = {345--378},
    doi = {10.1002/jgt.22830}
}

@article {lang2023improved,
    author = {Lang, Richard and Postle, Luke},
    title = {An improved bound for the linear arboricity conjecture},
    journal = {Combinatorica},
    volume = {43},
    year = {2023},
    number = {3},
    pages = {547--569},
    doi = {10.1007/s00493-023-00024-9}
}

@book{aigner1999proofs,
    author = {Aigner, Martin and Ziegler, G\"unter M.},
    title = {Proofs from {T}he {B}ook},
    edition = {Sixth},
    publisher = {Springer, Berlin},
    year = {2018},
    pages = {326},
    doi = {10.1007/978-3-662-57265-8}
}

@inproceedings{feige2014short,
    author = {Feige, Uriel and Ravi, R. and Singh, Mohit},
    title = {Short tours through large linear forests},
    booktitle = {Integer Programming and Combinatorial Optimization, IPCO 2014},
    series = {Lecture Notes in Computer Science},
    volume = {8494},
    pages = {273--284},
    publisher = {Springer},
    year = {2014},
    doi = {10.1007/978-3-319-07557-0_23}
}

@article{akiyama1980covering,
    author = {Akiyama, Jin and Exoo, Geoffrey and Harary, Frank},
    title = {Covering and packing in graphs. {III}. {C}yclic and acyclic invariants},
    journal = {Mathematica Slovaca},
    volume = {30},
    year = {1980},
    number = {4},
    pages = {405--417},
    url={http://dml.cz/dmlcz/136252}
}

@article{radhakrishnan1997entropy,
  title={An entropy proof of {B}regman's theorem},
  author={Radhakrishnan, Jaikumar},
  journal={Journal of Combinatorial Theory, Series A},
  volume={77},
  number={1},
  pages={161--164},
  year={1997},
  publisher={Elsevier},
  doi={10.1006/jcta.1996.2727}
}

@article{bregman1973some,
  title={Some properties of nonnegative matrices and their permanents},
  author={Br{\`e}gman, Lev Meerovich},
  journal={Soviet Mathematics-Doklady},
  volume={14},
  pages={945-949},
  year={1973},
  organization={Dokl, Soviet Math},
  url={https://www.mathnet.ru/eng/dan/v211/i1/p27}
}

@article{falikman1981proof,
  title={Proof of the van der {W}aerden conjecture regarding the permanent of a doubly stochastic matrix},
  author={Falikman, Dmitry I.},
  journal={Matematicheskie Zametki},
  volume={29},
  number={6},
  pages={931--938},
  year={1981},
  url={https://www.mathnet.ru/eng/mzm/v29/i6/p931}
}

@article{egorychev1981solution,
  title={The solution of van der {W}aerden's problem for permanents},
  author={Egorychev, Gregory P.},
  journal={Advances in Mathematics},
  volume={42},
  number={3},
  pages={299--305},
  year={1981},
  publisher={Academic Press},
  doi={10.1016/0001-8708(81)90044-X}
}

@article{jerrum2004polynomial,
  title={A polynomial-time approximation algorithm for the permanent of a matrix with nonnegative entries},
  author={Jerrum, Mark and Sinclair, Alistair and Vigoda, Eric},
  journal={Journal of the ACM (JACM)},
  volume={51},
  number={4},
  pages={671--697},
  year={2004},
  publisher={ACM New York, NY, USA},
  doi={10.1145/1008731.1008738}
}

@misc{gajjala2026counterexamples,
    author = {Gajjala, Rishikesh},
    title = {Counterexamples to an Extremal Conjecture for Random Cycle-Factors},
    year = {2026},
    arxiv = {2604.26101}
}

\end{document}